\documentclass[12pt]{article} %{amsart}%

\usepackage{amsfonts}
\usepackage{graphics}
\usepackage{amssymb}

\newtheorem{theorem}{Theorem}[section]
\newtheorem{lemma}{Lemma}[section]
\newtheorem{definition}{Definition}[section]
\newtheorem{proposition}[theorem]{Proposition}
\newtheorem{corollary}[theorem]{Corollary}

\newtheorem{remark}{Remark}[section]

\begin{document}

\begin{center}
{\Large   Time Dependent Tempered Generalized Functions and  It\^o's
Formula  \\}
\end{center}

\vspace{0.3cm}

\begin{center}

{\large Pedro Catuogno\footnote{Research partially supported by CNPQ 302704/2008-6.} .}\\

\textit{Departamento de
 Matem\'{a}tica, Universidade Estadual de Campinas, \\ F. 54(19) 3521-5921 � Fax 55(19)
3521-6094\\ 13.081-970 -
 Campinas - SP, Brasil}
\end{center}

\begin{center}

{\large
Christian Olivera\footnote{Research  supported by CAPES PNPD N02885/09-3.}}\\
\textit{Departamento de Matem\'{a}tica- Universidade Federal de S\~ao Carlos
Rod. Washington Luis, Km 235 - C.P. 676 - 13565-905 S\~ao Carlos, SP - Brasil
\\ Fone: 55(016)3351-8218
\\  e-mail:  colivera@dm.ufscar.br}
\end{center}

\vspace{0.3cm}
\begin{center}
\begin{abstract}
\noindent The paper introduces  a novel  It\^o's formula for time
dependent tempered generalized functions. As an application, we
study the heat equation when initial conditions are allowed to be a
generalized tempered function. A new proof of the Ustunel- It\^o's
formula for tempered distributions is also provided.
\end{abstract}
\end{center}


\vspace{0.3cm} \noindent {\bf MSC2000 subject classification:}
 46F10, 46F30, 60H99 .

\section {Introduction}
\noindent The study of Stochastic Partial Differential Equations
(SPDE) in algebras of generalized functions could be traced back to
the early 1990s, see \cite{albe},
\cite{albe2}, \cite{ober4} and \cite{Russo}. The interest for singular stochastic processes,
such as the white noise process, and solve differential equations driven by these type of processes are the main reasons for
treating of SPDE in the framework of the algebras of generalized functions (see by example \cite{albe} and \cite{CO2}).

 \noindent The central theme of the present paper is to  develop stochastic calculus  via regularization  in the
 setting of algebras of generalized functions. %In the first step,  we  to clarify that mining the  It\^o's formula in this context.
The  main technique is  the construction of an algebra of tempered
generalized functions   via the  regularization scheme induced by expansions in  Hermite functions.
 This approach allows us to obtain an It\^o's formula for   elements in the algebra. In
particular, we deduce  the  Ustunel- It\^o's  formula (see
\cite{ustunel})  for tempered distributions (see also \cite{Kunita}
and \cite{rajeev}). Sections 2 and 3 of the present paper present
the relevant precise definitions and details.

\noindent  As an application of our results, we also show the
existence and uniqueness of the solution to the heat equation
 with the initial condition being a tempered  generalized function; this makes crucial use of the obtained
 It\^o's  formula for tempered generalized functions. See Sections 4 for details.

\newpage
\section{Generalized functions}
\subsection{Tempered distributions}

Let $\mathcal{S}(\mathbb{R}^{d})$ be the Schwartz
space on $\mathbb{R}^{d}$ i.e. the space of rapidly decreasing smooth real valued functions on $\mathbb{R}^{d}$.

We make use of the multi-index notation; a multi-index is a sequence
$\alpha = (\alpha_1,...,\alpha_d ) \in \mathbb{N}_0^d$ where
$\mathbb{N}_0$ is the set of nonnegative integers. The sum $|\alpha|
= \sum_{j=1}^n \alpha_j$ is called the order of $\alpha$. For every
multi-index $\alpha$ we write
\[
x^{\alpha}=x_1^{\alpha_1} \cdots x_d^{\alpha_d}
\]
and
\[
\partial^{\alpha}=\partial^{\alpha_1}_1 \cdots \partial^{\alpha_d}_d
\]
where $\partial_j=\frac{\partial}{\partial x_j}$.

The Schwartz topology on $\mathcal{S}(\mathbb{R}^{d})$ is given
by the family of seminorms
\[
\|f\|_{\alpha,\beta}=(\int_{\mathbb{R}^{d}} | x^{ \alpha }\partial^{\beta}f(x)  |^2 dx)^{\frac{1}{2}}
\]
where $\alpha,\beta \in \mathbb{N}_0^d$.

The Schwartz space $\mathcal{S}^{\prime}(\mathbb{R}^{d})$ of
tempered distributions is the dual space of
$\mathcal{S}(\mathbb{R}^{d})$.

The {\it{Hermite polynomials}} $H_n(x)$ are defined by
\begin{equation}\label{forpolHe}
    H_n(x)=(-1)^n e^{\frac{x^2}{2}} \frac{d^n}{dx^n}  e^{-\frac{x^2}{2}}
\end{equation}

%or equivalently
%\begin{equation}\label{forpolHp}
 %     H_n(x) =2^{-\frac{n}{2}}\sum_{k=0}^{[n/2]}\frac{(-1)^k
 %     n!( \sqrt{2}x)^{n-2k}}{k!(n-2k)!}.
%\end{equation}
and the {\it{Hermite functions}} $h_n(x)$ are defined by
\begin{equation}\label{relfunpolh}
    h_{n}(x)=  (\sqrt{2\pi}n!)^{-\frac{1}{2}}
    e^{{-\frac{1}{4}x^2}}H_n(x)
\end{equation}
for $n\in \mathbb{N}_{0}$.

The $\alpha$-th. Hermite function on $\mathbb{R}^d$ is given by
\[
h_{\alpha}(x_1,...,x_d)=h_{\alpha_{1}}(x_1) \cdots h_{\alpha_{d}}(x_d)
\]
where $\alpha=(\alpha_1,...,\alpha_d) \in \mathbb{N}_{0}^{d}$.

The Hermite functions are in the Schwartz
space on $\mathbb{R}^d$ and the set $\{ h_{\alpha}: \alpha \in \mathbb{N}_{0}^{d} \}$ is an orthonormal basis for $L^{2}(\mathbb{R}^{d})$.

We consider the directed family of norms $\{ | \cdot |_{n} : n \in \mathbb{N}_{0} \}$ on $\mathcal{S}(\mathbb{R}^d)$, given by

\[
|\varphi|^2_{n}:=\sum_{\beta\in
\mathbb{N}_{0}^d}(2|\beta|+d)^{2n}(\int_{\mathbb{R}^d} \varphi(x)h_{\beta}(x)dx)^2.
\]

We observe that the families of seminorms $  \{
|\cdot |_{n} : n \in \mathbb{N}_{0}\}$ and $ \{\| \cdot \|_{\alpha,\beta}  : \alpha,\beta\in
\mathbb{N}_{0}^{d} \}$  on
$\mathcal{S}(\mathbb{R}^{d})$ are equivalent.

Let $x \in \mathbb{R}^d$, and  denote by $\tau_x$ the translation
operator defined on functions by the formula $\tau_x
\varphi(y)=\varphi(y-x)$. It follows immediately that
$\tau_x(\mathcal{S}(\mathbb{R}^d))\subset \mathcal{S}(\mathbb{R}^d)$
and that $\tau_{-x}$ is the inverse of $\tau_x$, thus, we can
consider $\tau_x$ acting on the tempered distributions by
\[
\tau_x T(\varphi)=T( \tau_{-x}\varphi).
\]
\begin{lemma}\label{lema}
Let $n \in \mathbb{N}_0$, there exists a polynomial $P_{n}$ with
nonnegative coefficients such that for all $x \in \mathbb{R}^d$,
\begin{equation}
|\tau_x \varphi |_{n} \leq P_{n}(|x|) | \varphi |_{n}
\end{equation}
for all $\varphi \in \mathcal{S}(\mathbb{R}^d)$.
\end{lemma}
\begin{proof}
Using Proposition 3.3 from  \cite{rt}, for each $n \in \mathbb{N}_0$
there exist constants $C_1 (n)$ and $C_2 (n)$ such that
\[
|\varphi |_{n}  \leq C_1(n) \sum_{|\alpha|, |\beta| \leq 2n} \| \varphi \|_{\alpha, \beta}
 \leq C_2(n)|\varphi|_n
\]
for all $\varphi \in \mathcal{S}(\mathbb{R}^d)$.

Thus
\begin{eqnarray*}
|\tau_x \varphi |_{n}  & \leq & C_1(n)
\sum_{|\alpha|, |\beta| \leq 2n}(\int_{\mathbb{R}^d}y^{2\alpha}(\partial^{\beta}\tau_x\varphi)^2(y) dy)^{\frac{1}{2}} \\
& = & C_1(n) \sum_{|\alpha|, |\beta| \leq 2n}(\int_{\mathbb{R}^d}(y+x)^{2\alpha}(\partial^{\beta}\varphi)^2(y) dy)^{\frac{1}{2}} \\
& \leq & P_n(|x|)|\varphi|_n,
\end{eqnarray*}
where $P_n$ is a polynomial of degree lower or equal to $2n$.
\end{proof}

Multiplication on  $\mathcal{S}(\mathbb{R}^d)$ has the following
property: for all $n \in \mathbb{N}_0$ there exists $r,s \in
\mathbb{N}_0$ and $C_n \in \mathbb{R}$ such that
\begin{equation}
|\varphi \psi|_n \leq C_n |\varphi|_r |\psi|_s
\end{equation}
for all $\varphi, \psi \in \mathcal{S}(\mathbb{R}^d)$ (see for
instance \cite{RTR} ). We shall make often use of this property.

The Hermite representation theorem for
$\mathcal{S}(\mathbb{R}^{d})$
($\mathcal{S}^{\prime}(\mathbb{R}^{d})$) states a topological
isomorphism beetwen  $\mathcal{S}(\mathbb{R}^{d})$
($\mathcal{S}^{\prime}(\mathbb{R}^{d})$) and the space of sequences
$\mathbf{s}_{d}$ ($\mathbf{s}_{d}^{\prime}$).

Let $\mathbf{s}_{d}$ be the space of sequences
\[
\mathbf{s}_{d}=\{(a_\beta)\in
\ell^2(\mathbb{N}^{d})):\sum_{\beta \in \mathbb{N}_{0}^{d}}(2|\beta|+d)^{2n}\mid
a_{\beta}\mid^2<\infty, \; \mbox{for all } \; n \in
\mathbb{N}_{0}\}.
\]

The space $\mathbf{s}_d$ is a locally convex space with the
family of norms
\[
| (a_\beta)|_n = (\sum_{\beta \in \mathbb{N}_{0}^{d}}(2|\beta|+d)^{2n}\mid
a_{\beta}\mid^2)^{\frac{1}{2}},
\]
where $n \in \mathbb{N}_{0}$.

The topological dual space to $\mathbf{s}_{d}$, denoted by
$\mathbf{s}_{d}^{\prime}$, is given by
\[
\mathbf{s}_{d}^{\prime}=\{(b_{\beta}): \mbox{for some
}\;(C,m)\in\mathbb{R}\times \mathbb{N}_{0}^{d}, \; \mid
b_{\beta} \mid \leq C(2|\beta|+d)^{m} \mbox{ for all } \beta \},
\]
and the natural pairing of elements from $\mathbf{s}_{d}$
and $\mathbf{s}_{d}^{\prime}$, denoted by $\langle \cdot, \cdot
\rangle$, is given by
\[
\langle (b_{\beta}), (a_{\beta}) \rangle = \sum_{\beta \in
\mathbb{N}_{0}^{d}} b_{\beta}a_{\beta},
\]
for $(b_{\beta}) \in \mathbf{s}_{d}^{\prime}$ and $(a_{\beta}) \in
\mathbf{s}_{d}$.

It is clear that $\mathbf{s}_{d}^{\prime}$ is an algebra with the
 pointwise operations:
\begin{eqnarray*}
(b_{\beta})+(b^{\prime}_{\beta})& = & (b_{\beta}+b^{\prime}_{\beta}) \\
(b_{\beta})\cdot(b^{\prime}_{\beta})& = &
(b_{\beta}b^{\prime}_{\beta}),
\end{eqnarray*}
and $\mathbf{s}_{d}$ is an ideal of $\mathbf{s}_{d}^{\prime}$.

%We define the ring of $d$-tempered numbers as the quotient $\mathbf{h}_d=\mathbf{s}^{\prime}_d/\mathbf{s}_d$ and denote by $[b_\beta]$ the equivalence class of $(b_\beta) \in \mathbf{s}_{d}^{\prime}$.

%We observe that $\mathbf{h}_d$ is not a field, since there exists zero divisors in $\mathbf{h}_d$. We have that $\iota : \mathbb{R} \rightarrow \mathbf{h}_d$ defined by
% $\iota(x)=[x]$ is an embedding of rings.

%\noindent The relation between $\mathbf{s}_{d}$
%($\mathbf{s}_{d}^{\prime}$) and $\mathcal{S}(\mathbb{R}^{d})$
%($\mathcal{S}^{\prime}(\mathbb{R}^{d})$) is induced by the Hermite
%functions, via Hermite coefficients (evaluation).

\begin{theorem}[N-representation theorem for $\mathcal{S}(\mathbb{R}^{d})$ and $\mathcal{S}^{\prime}(\mathbb{R}^{d})$]
\label{rpreS} $\mathbf{a)}$ Let
$\mathbf{h}:\mathcal{S}(\mathbb{R}^{d})\rightarrow \mathbf{s}_{d}$
be the application
\[
\mathbf{h}(\varphi)=(\int \varphi (x)h_{\beta}(x) dx).
\]
Then $\mathbf{h}$ is a topological isomorphism. Moreover,
\[
|\mathbf{h}(\varphi)|_n=|\varphi|_n
\]
for all $\varphi \in \mathcal{S}(\mathbb{R}^{d})$.

\noindent $\mathbf{b)}$ Let
$\mathbf{H}:\mathcal{S}^{\prime}(\mathbb{R}^{d})\rightarrow
\mathbf{s}_{d}^{\prime}$ be the application
$\mathbf{H}(T)=(T(h_{\beta}))$. Then $\mathbf{H}$ is a
topological isomorphism. Moreover, if $T \in
\mathcal{S}^{\prime}(\mathbb{R}^{d})$ we have that
\[
T=\sum _{\beta \in \mathbb{N}_0^{d}} T(h_{\beta})h_{\beta}
\]
in the weak sense and for all $\varphi \in
\mathcal{S}(\mathbb{R}^{d})$,
\[
T(\varphi)=\langle \mathbf{H}(T),\mathbf{h}(\varphi)\rangle.
\]
\end{theorem}

\begin{proof}
See  for instance \cite{Red1} pp. 143 or \cite{Schw} pp. 260.
\end{proof}

The sequences $\mathbf{h}(\varphi)$ and $\mathbf{H}(T)$ will be
referred to as  the {\it Hermite coefficients} of the tempered
function $\varphi$ and the tempered  distribution $T$ respectively.

\begin{corollary}\label{coro}
 For every $T \in \mathcal{S}^{\prime}(\mathbb{R}^{d})$ there exists $n \in
\mathbb{N}_0$, such that
\[
\|T \|^2_{-n}:=\sum_{\beta \in \mathbb{N}^d_0}(2|\beta|+d)^{-2n}T(h_\beta)^2< \infty.
\]
\end{corollary}
\begin{proof}
By Theorem \ref{rpreS}, $(T(h_{\beta})) \in \mathbf{s}_d^{\prime}$. Thus,
there exists $(C, l) \in \mathbb{R}\times \mathbb{N}_0$ such that
$| T(h_{\beta})| \leq C(2|\beta| + d)^l$ for all $\beta \in \mathbb{N}^d_0$. Now,
taking $n = l+1$ the Corollary follows.
\end{proof}

\subsection{Tempered generalized functions}

The aim of this subsection is to give an extension to the multidimensional case of the theory of tempered generalized  functions
introduced  by the authors in \cite{CO1}. %We introduce the algebra $\mathcal{H}_T(\mathbb{R}^d)$ of time dependent tempered generalizedfunctions.
Let $\mathcal{S}^1_T(\mathbb{R}^d)$ be the set of functions $f:[0,T]\times\mathbb{R}^d
\rightarrow \mathbb{R}$ such that for each $t \in [0,T]$,
$f(t,\cdot) \in \mathcal{S}(\mathbb{R}^d)$ and for each $x \in \mathbb{R}^d$,
$f(\cdot, x) \in C^1([0,T])$. It is clear that
$\mathcal{S}^1_T(\mathbb{R}^{d})^{\mathbb{N}^d_0}$ has the structure of an associative,
commutative differential algebra with the natural operations:
\begin{eqnarray*}
(f_{\beta})+(g_{\beta})& := & (f_{\beta}+ g_{\beta}) \\
a(f_{\beta})& := & (af_{\beta}) \\
(f_{\beta})\cdot (g_{\beta})& := & (f_{\beta} g_{\beta}) \\
\partial^{\alpha}(f_{\beta})& := & (\partial_x^{\alpha}f_{\beta}) \mbox{ for each $\alpha \in \mathbb{N}_0^d$.}
\end{eqnarray*}
In order to define the $1$-time dependent tempered generalized functions, we consider $\mathcal{H}_{T,1,d}^{\prime}$
the subalgebra of $\mathcal{S}^1_T(\mathbb{R}^{d})^{\mathbb{N}^d_0}$ given by
\[
\{(f_{\beta})\in \mathcal{S}^1_T(\mathbb{R}^d)^{\mathbb{N}^d_0}  :  \mbox{for each $n
\in \mathbb{N}_0$,} \ (\sup_{t \in [0,T]}| f_{\beta}(t, \cdot)
|_{n}),\ (\sup_{t \in [0,T]}| \frac{\partial f_{\beta}}{\partial
t}(t, \cdot) |_{n})\in \mathbf{s}_d^{\prime}  \}
\]
and $\mathcal{H}_{T,1,d}$ its differential ideal given by
\[
\{(f_{\beta})\in \mathcal{S}^1_T(\mathbb{R}^d)^{\mathbb{N}^d_0}  :  \mbox{for each $n
\in \mathbb{N}_0$,} \ (\sup_{t \in [0,T]}| f_{\beta}(t, \cdot)
|_{n}),\ (\sup_{t \in [0,T]}| \frac{\partial f_{\beta}}{\partial
t}(t, \cdot) |_{n})\in \mathbf{s}_d  \}.
\]

\noindent  The $1$-time dependent tempered algebra on $\mathbb{R}^d$ is
defined by

\[
\mathcal{H}^1_T(\mathbb{R}^d):=\mathcal{H}_{T,1,d}^{\prime}/\mathcal{H}_{T,1,d}
\]

\noindent The elements of $\mathcal{H}^1_T(\mathbb{R}^d)$ are called $1$-time dependend tempered generalized
functions. Let $(f_\beta) \in \mathcal{H}_{T,1,d}^{\prime}$ we shall  use $[f_\beta]$ to
denote the equivalent class $(f_\beta) + \mathcal{H}_{T,1,d}$.

\begin{remark}
In order to introduce the $0$-dependent tempered generalized functions, we consider
$\mathcal{S}^0_T(\mathbb{R}^d)$ the of the set of functions $f:[0,T]\times\mathbb{R}^d
\rightarrow \mathbb{R}$ such that for each $t \in [0,T]$,
$f(t,\cdot) \in \mathcal{S}(\mathbb{R}^d)$ and for each $x \in \mathbb{R}^d$,
$f(\cdot, x) \in C([0,T])$. Let $\mathcal{H}_{T,0,d}^{\prime}$ be the subalgebra given by
\[
\{(f_{\beta})\in \mathcal{S}^0_T(\mathbb{R}^d)^{\mathbb{N}^d_0}  :  \mbox{for each $n
\in \mathbb{N}_0$,} \ (\sup_{t \in [0,T]}| f_{\beta}(t, \cdot)
|_{n}) \in \mathbf{s}_d^{\prime}  \}
\]
and $\mathcal{H}_{T,0,d}$ its differential ideal given by
\[
\{(f_{\beta})\in \mathcal{S}_T(\mathbb{R}^d)^{\mathbb{N}^d_0}  :  \mbox{for each $n
\in \mathbb{N}_0$,} \ (\sup_{t \in [0,T]}| f_{\beta}(t, \cdot)
|_{n}) \in \mathbf{s}_d  \}.
\]

\noindent  The $0$-time dependent  tempered algebra on $\mathbb{R}^d$ is
defined by

\[
\mathcal{H}^0_T(\mathbb{R}^d):=\mathcal{H}_{T,0,d}^{\prime}/\mathcal{H}_{T,0,d}
\]

\end{remark}

\begin{remark}
In a similar way we can define the tempered algebra
\[
\mathcal{H}(\mathbb{R}^d):= \mathcal{H}_{d}^{\prime} /\mathcal{H}_{d}
\]
\noindent where
\[
\mathcal{H}_{d}^{\prime}:=\{(f_{\beta})\in \mathcal{S}(\mathbb{R}^d)^{\mathbb{N}^d_0}  :  \mbox{for each $n
\in \mathbb{N}_0$,} \ ( |  f_{\beta} |_{n} ) \in \mathbf{s}_d^{\prime}  \}
\]
\noindent and
\[
 \mathcal{H}_{d} := \{(f_{\beta})\in \mathcal{S}(\mathbb{R}^d)^{\mathbb{N}^d_0}  :  \mbox{for each $n
\in \mathbb{N}_0$,} \  (| f_{\beta}|_{n})\in \mathbf{s}_d  \}.
\]
\noindent The
elements of $\mathcal{H}(\mathbb{R}^d)$ are called tempered generalized
functions.
\end{remark}

\begin{proposition}
\begin{enumerate}
\item $\mathcal{H}(\mathbb{R}^d)$ is a subalgebra of $\mathcal{H}^0_T(\mathbb{R}^d)$ $(\mathcal{H}^1_T(\mathbb{R}^d))$.
\item Let $[f_{\beta}] \in
\mathcal{H}^1_T(\mathbb{R}^d)$. Then
\[
\frac{\partial}{\partial t}[f_{\beta}(t,
\cdot )]:=[\frac{\partial f_{\beta}}{\partial t}(t, \cdot )]
\in \mathcal{H}^0(\mathbb{R}^d)
\]
for every $t \in [0,T]$.
\item Let $h \in C^1([0,T])$ and $[f_{\beta}] \in
\mathcal{H}(\mathbb{R}^d)$. Then $h[f_{\beta}]:=[hf_{\beta}] \in
\mathcal{H}^1_T(\mathbb{R}^d)$.
\end{enumerate}
\end{proposition}
\begin{proof}
The proofs are straigthforward from the definitions.
\end{proof}

We observe that there exists a natural linear embedding $\iota:\mathcal{S}^{\prime}(\mathbb{R}^d)\rightarrow \mathcal{H}(\mathbb{R}^d)$, given by
\[
\iota(T)=[T_\beta],
\]
where $T_\beta=\sum_{\gamma \leq \beta}T(h_\gamma)h_\gamma$. Moreover, we have that

$\mathbf{a)}$ For all $\varphi \in \mathcal{S}(\mathbb{R}^d)$, $\iota(\varphi)=[\varphi]$,

$\mathbf{b)}$ For all $\varphi,\psi\in \mathcal{S}(\mathbb{R}^d)$,
$\iota(\varphi\psi)=\iota(\varphi) \cdot \iota(\psi)$,

$\mathbf{c)}$ For all $\alpha \in \mathbb{N}_0^d$, $\iota \circ \partial^{\alpha}=\partial^{\alpha} \circ \iota$.

\noindent  The translation operator $\tau_x: \mathcal{H}^0_T(\mathbb{R}^d)\rightarrow \mathcal{H}^0_T(\mathbb{R}^d)$ ($x \in \mathbb{R}^d$) is defined by
\[
 \tau_x[f_\beta]:=[\tau_xf_\beta].
\]
\noindent It  follows from Lemma 1.1 that $\tau_x$ is well defined. Analogously, $\tau_x: \mathcal{H}^1_T(\mathbb{R}^d)\rightarrow \mathcal{H}^1_T(\mathbb{R}^d)$
 ($x \in \mathbb{R}^d$) is well defined.

\noindent In the algebra $\mathcal{H}(\mathbb{R}^d)$ we have a weak
equality, namely, the association of tempered generalized functions.
More precisely, we say that the tempered generalized functions
$[f_\beta]$ and $[g_\beta]$ are associated, and denote this
association by $[f_\beta] \approx [g_\beta]$, if for all $\varphi
\in \mathcal{S}(\mathbb{R}^d)$,
\[
\lim_{\beta \rightarrow \infty}\int_{\mathbb{R}^n}(f_\beta(x)-g_\beta(x))\varphi(x)dx=0.
\]
We observe that $\approx$ is a equivalence relation on $\mathcal{H}(\mathbb{R}^d)$.

\begin{proposition}\label{p2}
\begin{enumerate}
\item Let $T$ be a tempered distribution and $x \in \mathbb{R}^d$. Then $\iota(\tau_xT) \approx \tau_x\iota(T)$.
\item Let $T$ and $S$ be tempered distributions such that $\iota(T) \approx \iota(S)$. Then $T=S$.
\end{enumerate}
\end{proposition}
\begin{proof}
1) Let $\varphi \in \mathcal{S}(\mathbb{R}^d)$; applying Theorem
\ref{rpreS} and elementary properties of the translation, we obtain
\begin{eqnarray*}
\lim_{\beta \rightarrow \infty}\int_{\mathbb{R}^d}\tau_xT_\beta(y)\varphi(y)dy & = & \lim_{\beta \rightarrow \infty}\int_{\mathbb{R}^d}T_\beta(y)\tau_{-x}\varphi(y)dy \\
 & = & T(\tau_{-x}\varphi) \\
 & = & \tau_xT(\varphi) \\
 & = &  \lim_{\beta \rightarrow \infty}\int_{\mathbb{R}^n}(\tau_xT)_\beta(y)\varphi(y)dy.
\end{eqnarray*}

2) For all $\alpha \in \mathbb{N}^d_0$,
\[
T(h_\alpha)-S(h_\alpha)=\lim_{\beta \rightarrow \infty}T_\beta(h_\alpha)-S_\beta(h_\alpha)=0.
\]
Theorem \ref{rpreS}  implies that  $T=S$.
\end{proof}

\begin{remark} We would like to  recall that he algebra  $\mathcal{H}(\mathbb{R}^d)$ is identical to (a sequential version of) the space
$G_{S(\mathbb{R}^d)}$, it has been studied extensively by G. Garetto
and their coauthors (see \cite{gare1}, \cite{gare2} and
\cite{gare3}). They introduced the space $G_{S(\mathbb{R}^d)}$  via
the family of seminorms $\{ \|\cdot\|_{\alpha, \beta, \infty}:
\alpha, \beta \in \mathbb{N}^d_0\}$. Notice that it does not matter
which family of seminorms is being used in this definition, as long
as it generates
the same locally convex topology (see \cite{gare1}).\\
Our approach differs, we have introduced the
 algebra $\mathcal{H}(\mathbb{R}^d))$ via the seminorms $\| \ \|_{m}$  since we are thinking in approximations of distributions
 (induced by Hilbert spaces) in terms of  orthogonal series in contrast with the classic theory where the approximation is done by convolution.
See \cite{CO2} for an application of this idea to stochastic distributions.
\end{remark}

\begin{remark} We would like to  mention that other   general properties of $\mathcal{H}(\mathbb{R}^d)$ can be
studied in this setting. For example the concepts of point value,
integral and Fourier transform for elements in
$\mathcal{H}(\mathbb{R}^d)$ can be defined,  see  \cite{CO1} for the
one-dimensional case.
   \end{remark}

\section{ It\^o's formula for tempered generalized functions}

\noindent Let $(\Omega, \mathcal{F}, \{ \mathcal{F}_t: t \in [0,T]
\}, \mathbb{P})$ be a filtered probability space, which satisfies
the usual hypotheses. For a recent account of stochastic calculus
we refer the reader to the book of  Ph. Protter \cite{Protter}.

\begin{definition}
Let $X$ be a $\mathbb{R}^d$ valued continuous jointly measurable
process, $V$  a continuous finite variation process and $[f_\beta]
\in \mathcal{H}^0_T(\mathbb{R}^d)$. Define the integral of
$\tau_X[f_\beta]$ in relation to $V$, from $0$ to $t$,  and  denoted
by $\int_0^t \tau_{X_s}[f_\beta]dV_s$, by:
\[
[\int_0^t\tau_{X_s}f_\beta(s, \cdot) dV_s],
\]
where the integral is given in the sense of Bochner-Stieltjes.
\end{definition}
For each $\omega \in \Omega$ and $t \in [0,T]$, we have that
$[\int_0^t\tau_{X_s}f_\beta(s,\cdot) dV_s(\omega)]$ is well defined
as an element of $\mathcal{H}^0_T(\mathbb{R}^d)$. In fact, since
$\tau_{X_s(\omega)}f_\beta(s,\cdot) \in \mathcal{S}^0_T(\mathbb{R}^d)$
and making use of definitions and the Lemma \ref{lema} we see that
\begin{eqnarray}\label{for1}
|\int_0^t \tau_{X_s}f_\beta(s, \cdot) dV_s(\omega)|_n & \leq & \int_0^t |\tau_{X_s(\omega)}f_\beta(s, \cdot)|_n d|V|_s(\omega) \nonumber\\
 & \leq & ( \int_0^t P_n(|X_s(\omega)|) d|V|_s(\omega) ) \sup_{s \in [0,T]}|f_\beta(s, \cdot)|_n
\end{eqnarray}
where $|V|_t(\omega)$ is the total variation of $V$ in $[0,t]$.

\begin{theorem}
Let $f=[f_\beta] \in \mathcal{H}^1_T(\mathbb{R}^d)$ and $X=(X^1,...,X^d)$ be a $\mathbb{R}^d$ valued continuous
semimartingale. Then
\begin{eqnarray}\label{ito}
\tau_{X_t}f & = & \tau_{X_0}f + \int_0^t \partial_t\tau_{X_s}f(s, \cdot)ds-\int_0^t\nabla\tau_{X_s}f \cdot dX_s \nonumber \\
 & & +\frac{1}{2}\sum_{i,j=1}^d\int_0^t\partial_{ij}\tau_{X_s}fd \langle X^i, X^j \rangle_s
\end{eqnarray}
where $\int_0^t\nabla\tau_{X_s}f \cdot dX_s$ is defined as $[\sum_{i=1}^d\int_0^t\partial_i\tau_{X_s}f_\beta dX_s^i]$.
\end{theorem}
\begin{proof}
Applying the classical It\^o's formula to $f_\beta$, we have
\begin{eqnarray}\label{ito1}
\tau_{X_t}f_\beta(t, x)& = & \tau_{X_0}f_\beta(0,x) +\int_0^t \partial_t\tau_{X_s} f_{\beta}(s, x)ds - \sum_{i=1}^d\int_0^t\partial_i\tau_{X_s}f_\beta(s, x) dX_s^i \nonumber \\ & & +
 \frac{1}{2}\sum_{i,j=1}^d\int_0^t\partial_{ij}\tau_{X_s}f_\beta(s, x)~ d\langle X^i, X^j\rangle_s.
\end{eqnarray}

Taking equivalent classes in (\ref{ito1}), we obtain that
$\int_0^t\nabla\tau_{X_s}f \cdot dX_s \in \mathcal{H}^0_T(\mathbb{R}^d)$
and hence (\ref{ito})~ holds in $\mathcal{H}^0_T(\mathbb{R}^d)$.
\end{proof}

\begin{corollary}\label{ito2}
Let $f=[f_\beta] \in \mathcal{H}(\mathbb{R}^{d})$ and $X=(X^1,...,X^d)$ be a continuous
semimartingale. Then
\begin{eqnarray*}
\tau_{X_t}f & = & \tau_{X_0}f- \int_0^t\nabla\tau_{X_s}f \cdot dX_s +\frac{1}{2}\sum_{i,j=1}^d\int_0^t\partial_{ij}\tau_{X_s}fd<X^i,X^j>_s.
\end{eqnarray*}
\end{corollary}

\subsection{It\^o's formula for tempered distributions}
In order to prove the \"{U}stunel-It\^o's formula for tempered
distributions (see \cite{ustunel}), we need the following result
from the stochastic integration theory in nuclear spaces (see
\cite{ustunel2}). Let $T_t$ be a
$\mathcal{S}^{\prime}(\mathbb{R}^d)$-valued continuous predictable
process and $X_t$ be a continuous semimartingale, then $\int_0^tT_s
dX_s$ is the unique $\mathcal{S}^{\prime}(\mathbb{R}^d)$ valued
semimartingale such that for all $\varphi \in
\mathcal{S}(\mathbb{R}^d)$,
\[
(\int_0^tT_s dX_s)(\varphi)=\int_0^tT_s(\varphi)dX_s.
\]
\begin{lemma}\label{lemma}
Let $T \in \mathcal{S}^{\prime}(\mathbb{R}^d)$,  $X$ be a $\mathbb{R}^d$ valued continuous semimartingale and $V$ be a continuous finite variation
process. Then
\[
\int_0^t\tau_{X_s}\iota(T)dV_s \approx \iota(\int_0^t \tau_{X_s}T dV_s).
\]
\end{lemma}

\begin{proof}
Let $\iota(T)=[T_\beta]$ and $\varphi \in \mathcal{S}(\mathbb{R}^d)$. Since $\lim_{\beta \rightarrow \infty}T_\beta=T$ we have
\[
\lim_{\beta \rightarrow \infty}\tau_{X_s(\omega)}T_\beta(\varphi)=\tau_{X_s(\omega)}T(\varphi)
\]
for all $s$ and $\omega \in \Omega$.

Applying the Corollary \ref{coro} and  Lemma \ref{lema} we have that there exists $q \in
\mathbb{N}_0$ such that
\begin{eqnarray*}
| \tau_{X_s(\omega)}T_\beta(\varphi) |& \leq &
 |T_\beta |_{-q}   |\tau_{X_s(\omega)}\varphi |_{q} \\
& \leq & P_q(|X_s(\omega)|) \ | T |_{-q}  \  |\varphi |_q.
\end{eqnarray*}
By the dominate convergence Theorem we obtain
\[
\lim_{\beta \rightarrow \infty} \int_0^t \tau_{X_s}T_\beta(\varphi) dV_s
=\int_0^t \tau_{X_s}T(\varphi) dV_s.
\]
We conclude that
\begin{eqnarray*}
\lim_{\beta \rightarrow \infty} \int_0^t \tau_{X_s(\omega)} \iota(T) dV_s(\omega)_\beta(\varphi) & = &
\lim_{\beta \rightarrow \infty} \int_0^t \tau_{X_s(\omega)}T_\beta(\varphi) dV_s(\omega) \\
& = & \int_0^t \tau_{X_s(\omega)}T(\varphi) dV_s(\omega) \\
& = & \lim_{\beta \rightarrow \infty} \iota (\int_0^t \tau_{X_s(\omega)}T  dV_s(\omega))_\beta(\varphi),
\end{eqnarray*}
and the proof is complete.

\end{proof}

\begin{proposition}[It\^o's formula for tempered distributions]
Let $T \in \mathcal{S}^{\prime}(\mathbb{R}^d)$ and $X=(X^1,...,X^d)$ be a continuous
semimartingale. Then
\[
\tau_{X_t}T=\tau_{X_0}T-\sum_{i=1}^d\int_0^t\partial_i\tau_{X_s}T dX_s^i +
\frac{1}{2}\sum_{i,j=1}^d\int_0^t\partial_{ij}\tau_{X_s}Td<X^i,X^j>_s.
\]
\end{proposition}

\begin{proof}
Applying It\^o's formula (\ref{ito})  to $\iota(T)$ and making use
of Proposition \ref{p2} and Lemma \ref{lemma} we deduce  that
\[
\iota (\tau_{X_t}T-\tau_{X_0}T-\frac{1}{2}\sum_{i,j=1}^d\int_0^t\partial_{ij}\tau_{X_s}fd<X^i,X^j>_s ) \approx \int_0^t\nabla\tau_{X_s}\iota (T) \cdot dX_s.
\]

According to an easy modification of Lemma \ref{lemma} (we make use
of the dominated convergence theorem for stochastic integrals) we
can to prove that

\begin{equation}\label{e1}
\int_0^t\nabla\tau_{X_s}\iota (T) \cdot dX_s \approx \iota(\sum_{i=1}^d \int_0^t\partial_i\tau_{X_s}T dX_s^i).
\end{equation}

Then, from the Proposition \ref{p2} and equality  (\ref{e1}) we conclude
the proof.
\end{proof}

\section{Heat equation in $\mathcal{H}_T^1(\mathbb{R}^d)$}

We introduce next the concept of expected value (or expectation) for
certain $\mathcal{H}(\mathbb{R}^d)$-valued random variables. More
precisely, let $X$ be a $\mathbb{R}^d$ valued random variable with
$\mathbb{E}(|X|^n) < \infty $ for all $n \in \mathbb{N}_0$ and   let
$f=[f_\beta] \in \mathcal{H}(\mathbb{R}^d)$. The expectation of
$\tau_X f$ is $\mathbb{E}(\tau_X f):=[\mathbb{E}(\tau_X f_\beta)]$.

We observe that $\mathbb{E}(\tau_X f)$ is  well-defined as an
element of $\mathcal{H}(\mathbb{R}^d)$. In fact, by Lemma
\ref{lema}, it follows that:
\begin{eqnarray*}
 |\mathbb{E}(\tau_X f_\beta)|_n & \leq & \mathbb{E}(P_n(|X|))
 |f_\beta|_n,
\end{eqnarray*}
for all $\beta \in \mathbb{N}_0^d$.

The remaining of the present section is concerned with the Cauchy
problem for the heat equation,

\begin{equation}\label{heat}
 \left \{
\begin{array}{lll}
u_t & = & \frac{1}{2} \triangle u \\
u_0 & = & f\in \mathcal{H}(\mathbb{R}^d).
\end{array}
\right .
\end{equation}

\begin{definition}
We say that $u \in \mathcal{H}^1_T(\mathbb{R}^d)$ is a generalized
solution of the Cauchy problem (\ref{heat}) if $u_t  =  \frac{1}{2} \triangle u$ in $\mathcal{H}^0_T(\mathbb{R}^d)$ and
$u_0  =  f$ in $\mathcal{H}(\mathbb{R}^d)$.
\end{definition}

\begin{proposition}\label{lemaheat}  For every   $f\in \mathcal{H}(\mathbb{R}^d) $ there
exist a unique solution to the Cauchy problem (\ref{heat}) in
$\mathcal{H}^1_T(\mathbb{R}^d)$.
\end{proposition}

\begin{proof} {\large Step 1} ( Existence) By It\^o's formula (\ref{ito}) we have

\begin{equation}\label{exit}
\tau_{B_t}f = f-\int_0^t\nabla\tau_{B_s}f \cdot dB_s +
\int_0^t \frac{1}{2} \triangle \tau_{B_s}fds.
\end{equation}

We observe that: $\mathbb{E}(\int_0^t \triangle \tau_{B_s}f
ds)=\int_0^t \mathbb{E}(\triangle \tau_{B_s}f) ds$; in fact, this is
a consequence of inequality (\ref{for1}),  that
$\mathbb{E}(|B_t|^n)< \infty$ and that $\mathbb{E}(\int_0^t|B_s|^nds)<
\infty$ for all $n \in \mathbb{N}_0$.

Taking expectation in (\ref{exit}) we obtain that
\[
\mathbb{E}(\tau_{B_t}f)=f+\int_0^t \frac{1}{2} \triangle \mathbb{E}(\tau_{B_s}f)ds.
\]
Thus $\mathbb{E}(\tau_{B_t}f)$ solves the Cauchy problem (\ref{heat}).

\noindent {\large Step 2} (Uniqueness)
    We consider the uniqueness. Suppose that $u=[u_\beta ]$ and $v=[v_\beta ]$ are two generalized
solutions of (\ref{heat}), we denote the difference $u_\beta - v_\beta$ by $a_\beta$ . By the definition, $a_\beta$
satisfies

\begin{equation}\label{heatuni}
 \left \{
\begin{array}{lll}
\frac{d}{dt}a_\beta & = & \frac{1}{2} \triangle a_\beta + h_\beta \\
a_\beta(0,\cdot) & = & g_\beta,
\end{array}
\right .
\end{equation}

\noindent with  $h_\beta \in \mathcal{H}_{T,0, d} $ and  $g_\beta \in \mathcal{H}_d$. Applying the Feynman–Kac formula (see \cite{Frei})
to $a_\beta$ we get

\begin{eqnarray}\label{uni2}
a_{\beta}(t,x) & = & \widetilde{\mathbb{E}}(g_\beta(x+\widetilde{B}_t) + \int_{0}^{t}    h_{\beta}(s,x+\widetilde{B}_s) \
 ds )
\end{eqnarray}
where $\widetilde{B}$ is a $d$-dimensional Brownian motion with $\widetilde{B}_0=0$ in an auxiliary
probability space.

It follows that $\sup_{t}\|  a_{\beta}(t,x) \|_{n} \in \mathbf{s}_d$, for each $n \in \mathbb{N}_0$. From this fact and  equation
(\ref{heatuni}) we have that $\sup_{t}\| \frac{d}{dt} a_{\beta}(t,x) \|_{n} \in \mathbf{s}_d$. We conclude that $(a_{\beta}) \in \mathcal{H}_{T,1,d}$ and thus (\ref{heat}) has an unique solution.
\end{proof}

\begin{remark}
\noindent The proof of existence and uniqueness of Proposition
\ref{lemaheat} can be  extended easily to the following Cauchy
problem,

\begin{equation}\label{heat2}
 \left \{
\begin{array}{lll}
u_t & = & \frac{1}{2} \triangle u + g,  \\
u_0 & = & f\in \mathcal{H}(\mathbb{R}^d)
\end{array}
\right .
\end{equation}

\noindent where  $ g\in \mathcal{H}_T^0(\mathbb{R}^d)$.
\end{remark}

\end{document}